\documentclass{aic}

\usepackage{color}
\usepackage{amsmath, amssymb, amsthm}
\usepackage{mathrsfs}
\usepackage{mathtools}
\usepackage[noabbrev,capitalize,nameinlink]{cleveref}
\crefname{equation}{}{}
\usepackage[noadjust]{cite}
\usepackage{graphics}
\usepackage{pifont}
\usepackage{tikz}
\usepackage{bbm}

\usetikzlibrary{arrows.meta}

\usepackage{environ}
\usepackage{framed}
\usepackage{url}
\usepackage[labelfont=bf]{caption}
\usepackage{cite}
\usepackage{graphicx}
\usepackage{tcolorbox}
\usepackage{enumerate}
\allowdisplaybreaks[1]

\numberwithin{equation}{section}
\newtheorem{theorem}{Theorem}[section]

\newtheorem{lemma}[theorem]{Lemma}

\crefname{claim}{Claim}{Claims}

\newtheorem{corollary}[theorem]{Corollary}

\newtheorem*{question*}{Question}

\theoremstyle{definition}

\newtheorem{question}[theorem]{Question}
\newtheorem*{definition*}{Definition}

\theoremstyle{remark}
\newtheorem*{remark}{Remark}

\newcommand{\sang}[1]{\langle #1 \rangle}

\newcommand{\mb}{\mathbb}

\newcommand{\mbm}{\mathbbm}
\newcommand{\mc}{\mathcal}

\newcommand{\on}{\operatorname}

\aicAUTHORdetails{%
  title = {Anticoncentration versus the Number of Subset Sums}, 
  author = {Vishesh Jain, Ashwin Sah, and Mehtaab Sawhney},
  plaintextauthor = {Vishesh Jain, Ashwin Sah, Mehtaab Sawhney},
  keywords = {anticoncentration},
}   

\aicEDITORdetails{%
   year={2021},
   number={6},
   received={19 January 2021},   
   published={28 June 2021},  
   doi={10.19086/aic.24872},      
}   

\begin{document}

\begin{frontmatter}[classification=text]

\title{Anticoncentration versus the Number of Subset Sums} 

\author[vj]{Vishesh Jain}
\author[as]{Ashwin Sah\thanks{Supported by NSF Graduate Research Fellowship Program DGE-1745302}}
\author[ms]{Mehtaab Sawhney\thanks{Supported by NSF Graduate Research Fellowship Program DGE-1745302}}

\begin{abstract}
Let $\vec{w} = (w_1,\dots, w_n) \in \mb{R}^{n}$. We show that for any $n^{-2}\le\epsilon\le 1$, if
\[\#\{\vec{\xi} \in \{0,1\}^{n}: \langle \vec{\xi}, \vec{w} \rangle = \tau\} \ge 2^{-\epsilon n}\cdot 2^{n}\]
for some $\tau \in \mathbb{R}$, then 
\[\#\{\langle \vec{\xi}, \vec{w} \rangle : \vec{\xi} \in \{0,1\}^{n}\} \le 2^{O(\sqrt{\epsilon}n)}.\]
This exponentially improves the $\epsilon$ dependence in a recent result of Nederlof, Pawlewicz, Swennenhuis, and W{\k{e}}grzycki and leads to a similar improvement in the parameterized (by the number of bins) runtime of bin packing. 
\end{abstract}
\end{frontmatter}


\section{Introduction} \label{sec:introduction}
For $\vec{w} := (w_1,\dots, w_n) \in \mb{R}^{n}$ and a real random variable $\xi$, recall that the L\'evy concentration function of $\vec{w}$ with respect to $\xi$ is defined for all $r\ge 0$ by
\[\mc{L}_{\xi}(\vec{w},r) = \sup_{\tau \in \mb{R}}\mb{P}[|w_1 \xi_1 + \dots + w_n \xi_n - \tau| \le r],\]
where $\xi_1,\dots, \xi_n$ are i.i.d.\ copies of $\xi$. In combinatorial settings (where $\vec{w} \in \mb{Z}^{n}$) a particularly natural and interesting case is when $r = 0$ and $\xi$ is a Bernoulli random variable, i.e., $\xi = 0$ with probability $1/2$ and $\xi = 1$ with probability $1/2$. For lightness of notation, we will denote this special case by
\[\rho(\vec{w}) = \mc{L}_{\on{Ber}(1/2)}(\vec{w},0) = \sup_{\tau\in\mb{R}}\mb{P}[\sang{\vec{w},\vec{\xi}}=\tau].\]
In this note, we study the following question.
\begin{question}
\label{question}
For a vector $\vec{w} = (w_1,\dots, w_n) \in \mb{R}^{n}$ with $\rho(\vec{w}) \ge \rho$, how large can the range
\[\mc{R}(\vec{w}) = \{w_1\xi_1 + \dots + w_n \xi_n: \xi_i \in \{0, 1\}\}\]
be?
\end{question}
The two extremal examples here are $\vec{w} = (0,0,\dots,0)$, which corresponds to $\rho(\vec{w}) = 1$, $|\mc{R}(\vec{w})| = 1$ and $\vec{w} = (1,10,\dots,10^{n-1})$, which corresponds to $\rho(\vec{w}) = 2^{-n}$, $|\mc{R}(\vec{w})| = 2^{n}$. Motivated by these examples, one may ask if there is a smooth trade-off between $\rho(\vec{w})$ and $|\mc{R}(\vec{w})|$. This turns out not to be the case. Indeed, for any $\epsilon > 0$, Wiman \cite{W17} gave an example of a $\vec{w} \in \mb{Z}^{n}$ for which $|\mc{R}(\vec{w})| \ge 2^{(1-\epsilon)n}$ and $\rho(\vec{w}) \ge 2^{-0.7447n}$. At the other end of the spectrum, when $\rho(\vec{w}) \ge 2^{-\epsilon n}$, the so-called inverse Littlewood--Offord theory \cite{RV08, TV09, NV11} \emph{heuristically} suggests that $\vec{w}$ is essentially contained in a low-rank generalized arithmetic progression of `small' volume so that $|\mc{R}(\vec{w})|$ is also `small'. However, the number of `exceptional elements' in the inverse Littlewood--Offord theorems \cite{RV08, TV09, NV11} is unfortunately too large to be able to rigorously establish such a statement.   

Nevertheless, in a recent work on the parameterized complexity of the bin packing problem (see \cref{sub:application}), Nederlof, Pawlewics, Swennenhuis and W{\k{e}}grzycki \cite{NPSW20} showed that for any $\epsilon > 0$, 
\[\rho(\vec{w}) \ge 2^{-\epsilon n} \implies |\mc{R}(\vec{w})| \le 2^{\delta(\epsilon)n},\]
where
\begin{equation}
\label{eqn:delta-old}
\delta(\epsilon) = O\left(\frac{\log\log(\epsilon^{-1})}{\sqrt{\log(\epsilon^{-1})}}\right).
\end{equation}
In particular, $\delta(\epsilon) \to 0$ as $\epsilon \to 0$. Moreover, we must have $\delta(\epsilon) \ge (2-o(1))\epsilon$, as can be seen by considering
\[\vec{w} = (C_1,\dots,C_1, C_2,\dots, C_2,\dots, C_{n/k},\dots,C_{n/k}) \in \mb{R}^{n},\]
where each $C_i$ is repeated $k$ times, and $C_i$ is sufficiently small compared to $C_{i+1}$ for all $i$. Indeed, for such $\vec{w}$, we have $\rho(\vec{w}) = 2^{-(\frac{1}{2}+o_k(1))\frac{n}{k}\log_2{k}}$ while $|\mc{R}(\vec{w})| \le 2^{(1+o_k(1))\frac{n}{k}\log_2{k}}$.

We conjecture that this example is essentially the worst possible, so that $\delta(\epsilon) \le 2\epsilon$. We are able to show that
\begin{equation}
\label{eqn:delta-new}
\delta(\epsilon) = O(\sqrt{\epsilon}),
\end{equation}
thereby obtaining an exponential improvement over \cref{eqn:delta-old}. More precisely, 
\begin{theorem}\label{thm:main}
Let $\epsilon>0$. For any $n\ge\epsilon^{-1/2}$ and any $\vec{w}\in\mb{R}^n$ satisfying $\rho(\vec{w})\ge \exp(-\epsilon n)$, we have
\[|\mc{R}(\vec{w})|\le\exp(C_{\ref{thm:main}}\epsilon^{1/2}n),\]
where $C_{\ref{thm:main}}$ is an absolute constant. 
\end{theorem}
We prove this theorem in \cref{sec:proof}.

\subsection{Application to bin packing}
\label{sub:application}
The bin packing problem is a classic NP-complete problem whose decision version may be stated as follows: given $n$ items with weights $w_1,\dots, w_n \in [0,1]$ and $m$ bins, each of capacity $1$, is there a way to assign the items to the bins without violating the capacity constraints? Formally, is there a map $f:[n] \to [m]$ such that 
$\sum_{i \in f^{-1}(j)}w_i \le 1$ for all $j \in [m]$?

Bj\"orklund, Husfeldt, and Koivisto \cite{BHK09} provided an algorithm for solving bin packing in time $\tilde{O}(2^{n})$ where the tilde hides polynomial factors in $n$. It is natural to ask whether the base of the exponent may be improved at all i.e.\ is there a (possibly randomized) algorithm to solve bin packing in time $\tilde{O}(2^{(1-\epsilon)n})$ for some absolute constant $\epsilon > 0$?

In recent work, Nederlof, Pawlewics, Swennenhuis and W{\k{e}}grzycki \cite{NPSW20} showed that this is true provided that the number of bins $m$ is fixed. More precisely, they showed that there exists a function $\sigma: \mb{N} \to \mb{R}^{>0}$ and a randomized algorithm for solving bin packing which, on instances with $m$ bins, runs in time $\tilde{O}(2^{(1-\sigma(m))n})$, where $\tilde{O}$ hides polynomials in $n$ as well as exponential factors in $m$. Their analysis, which crucially relies on \cref{eqn:delta-old}, gives a very small value of $\sigma(m)$ satisfying
\begin{equation}
    \label{eqn:sigma-old}
    \sigma(m) \le 2^{-m^9}. 
\end{equation}
Using \cref{thm:main} instead of \cref{eqn:delta-old} in a black-box manner in the analysis of \cite{NPSW20}, we exponentially improve the bound on $\sigma(m)$.
\begin{corollary}
\label{cor:bin-packing}
With notation as above, the randomized algorithm of \cite{NPSW20} solves bin packing instances with $m$ bins in time $\tilde{O}(2^{(1-\sigma(m))n})$ with high probability, for $\sigma \colon \mb{N} \to \mb{R}^{> 0}$ satisfying
\begin{equation}
    \label{eqn:sigma-new}
    \sigma(m) = \tilde{\Omega}(m^{-12}),
\end{equation}
where $\tilde{\Omega}$ hides logarithmic factors in $m$.
\end{corollary}

\begin{remark}
This follows by noting that the function $f_C(m)$ in \cite[Section~3.6]{NPSW20} is $\tilde{\Theta}(m^{-2})$ so that $\delta$ in \cite[Section~3.6]{NPSW20} is $\tilde{\Theta}(m^{-3})$. With \cref{thm:main}, the function $\varepsilon(\delta)$ in the runtime analysis of \cite[Section~3.4]{NPSW20} satisfies $\varepsilon(\delta) = O(\delta^2)$. Therefore, the function $f_B(\delta)$ in the same section is $\tilde{O}(\delta^{4})$, which is $\tilde{\Omega}(m^{-12})$. Note that if one were able to establish the conjecturally optimal bound $\delta = O(\epsilon)$, this would lead to $f_B(\delta) = \tilde{O}(\delta^{2})$, thereby giving the quadratically better $\sigma(m) = \tilde{\Omega}(m^{-6})$.
\end{remark}

\subsection{Notation}\label{sub:notation}
We use big-$O$ notation to mean that an absolute multiplicative constant is being hidden. We use $\on{Ber}(1/2)$ to denote the balanced $\{0,1\}$ Bernoulli distribution and $\on{Bin}(k)$ to denote the binomial distribution on $k$ trials with parameter $1/2$. Recall that $\on{Bin}(k)$ is the sum of $k$ independent $\on{Ber}(1/2)$ random variables. Given a distribution $\mu$, we let $\mu^{\otimes n}$ denote the distribution of a random vector with $n$ independent samples from $\mu$ as its coordinates. We also use the following standard additive combinatorics notation: $C+D = \{c+d: c\in C,d\in D\}$ is the sumset (if $C,D$ are subsets of the same abelian group), and for a positive integer $k$, we let $k\cdot C = C+\cdots+C$ ($k$ times) be the iterated sumset. Finally, in some cases we will use the notation $\Sigma\cdot$ or $\int\cdot$ to denote that the expression in the sum or integral is the same as in the previous line to simplify the presentation of long expressions.

\subsection{Outline of the proof}
\label{sub:outline}
As in \cite{NPSW20}, the starting point of our proof is the following observation: let $A$ denote a fixed (but otherwise arbitrary) set of unique preimages for points in $\mc{R}(\vec{w})$ (hence, $|A| = |\mc{R}(\vec{w})|$) and let $B$ denote the the set of preimages of a value $\tau \in \mb{R}$ realising $\rho(\vec{w})$. Then (\cref{lem:injective}) for any $k\geq 1$, the map $A \times (k\cdot B) \to A + k\cdot B$ is a bijection. In particular, if $\vec{a}$ is sampled from the uniform distribution on $A$ and $\vec{b}_1,\dots, \vec{b}_k$ are independently sampled from the uniform distribution on $B$, then
\begin{align*}
    |A| &= |A| \cdot \mb{P}[\vec{a} + \vec{b}_1 + \dots  + \vec{b}_k \in \{0,\dots, k+1\}^{n}]\\
        &= |A| \cdot \sum_{\vec{x} \in \{0,\dots, k+1\}^{n}}\mb{P}[\vec{a} + \vec{b}_1 + \dots + \vec{b}_k = \vec{x}]\\
        &\leq |A| \cdot \sum_{\vec{x} \in \{0,\dots, k+1\}^{n}} \mb{P}[\vec{a} = \vec{a}(\vec{x})] \cdot \mb{P}[\vec{b}_1 + \dots + \vec{b}_k = \vec{x} - \vec{a}(\vec{x})]\\
        &\leq \sum_{\vec{x} \in \{0,\dots, k+1\}^{n}}\mb{P}[\vec{b}_1 + \dots + \vec{b}_k = \vec{x} - \vec{a}(\vec{x})] 
\end{align*}

In \cite{NPSW20}, the largeness of $B$ is exploited by finding, for every $a \in A$, a large subset of $B$ which is `balanced' (in a certain sense) with respect to $a$. Instead, we exploit the largeness of $B$ directly by using the observation that the density of the uniform measure on $B$ with respect to the uniform measure on $\{0,1\}^{n}$ is at most $2^{n}/|B| \leq 2^{\varepsilon n}$. In particular, if we let $\mu_{k}$ denote the measure on $k\cdot B$ induced by the product measure on $B \times \dots \times B$ via the map $(b_1,\dots, b_k) \mapsto b_1+\dots +b_k$ and if we let $\on{Bin}(k)^{\otimes n}$ denote the $n$-fold product of the $\on{Binomial}(k,1/2)$ distribution, then the density of $\mu_{k}$ with respect to $\on{Bin}(k)^{\otimes n}$ is at most $2^{k\varepsilon n}$. This allows us to replace the measure $\mu_{k}$ appearing in the last line of the above equation by $\on{Bin}(k)^{\otimes n}$, at the cost of a factor of $2^{k\varepsilon n}$. Thus,
\[|A| \leq 2^{k\epsilon n}\cdot \sum_{\vec{x} \in \{0,\dots,k+1\}^{n}}\mb{P}_{\vec{x} \sim \on{Bin}(k)^{\otimes n}}[\vec{x} - \vec{a}(\vec{x})]\]

The above expression is still complicated by the presence of the shift $\vec{a}(\vec{x})$, about which we have no information except that it lies in the set $A$. The key technical lemma in the proof is \cref{lem:sup-ratio}, which essentially allows us to remove this shift after paying a factor which depends on $|A|$. Ultimately, this gives an upper bound on the sum in terms of $|A|$ and $k$, which amounts to an upper bound on $|A|$ in terms of $k, \epsilon$, and $|A|$. Optimizing the value of the free parameter $k$ now gives the desired conclusion. 

\section{Proof of \texorpdfstring{\cref{thm:main}}{Theorem 1.2}}
\label{sec:proof}

We begin by recording the following key comparison bound, which will be proved at the end of this section.

\begin{lemma}\label{lem:sup-ratio}
Let $n\ge k \ge C_{\ref{lem:sup-ratio}}$, where $C_{\ref{lem:sup-ratio}}$ is a sufficiently large absolute constant and let $\delta > 0$. For any $A \subseteq \{0,1\}^{n}$ with $|A| \le \exp(\delta n)$, the following holds.  
Let $\vec{x},\vec{b}\sim\on{Bin}(k)^{\otimes n}$ be independent $n$-dimensional random vectors. Then, 
\[\mb{E}_{\vec{x}}\bigg[\sup_{\vec{a}\in A}\frac{\mb{P}_{\vec{b}}[\vec{b} = \vec{x}-\vec{a}]}{\mb{P}_{\vec{b}}[\vec{b} = \vec{x}]}\bigg]\le\exp\left(C_{\ref{lem:sup-ratio}}\left(\frac{1}{k}+\sqrt{\frac{\delta}{k}}\right)n\right).\]
\end{lemma}

Let $n$, $\epsilon$, and $\vec{w}$ be as in \cref{thm:main}. Let $\tau$ be such that $\mb{P}[\sang{\vec{w}, \vec{\xi}} = \tau] = \rho(\vec{w})$, where $\vec{\xi}$ is a random vector with i.i.d.\ $\on{Ber}(1/2)$ components. Let 
\[B = \{\vec{\xi} \in \{0,1\}^{n} : \sang{\vec{w}, \vec{\xi}} = \tau\}.\]
In particular, $|B|\ge \exp(-\epsilon n)\cdot 2^{n}$.
Let $|\mc{R}(\vec{w})| = \exp(\delta n)$. For each $r \in \mc{R}(\vec{w})$, let $\vec{\xi}(r)$ be a fixed (but otherwise arbitrary) element of $\{0,1\}^{n}$ such that $\sang{\vec{w},\vec{\xi}(r)} = r$. Let
\[A =\{\vec{\xi}(r) \in \{0,1\}^{n}: r \in \mc{R}(\vec{w})\}.\]
Note that, by definition, for any distinct $\vec{a}_1, \vec{a}_2 \in A$, we have that $\sang{\vec{w}, \vec{a}_1} \neq \sang{\vec{w}, \vec{a}_2}$ and that $|A| = |\mc{R}(\vec{w})| = \exp(\delta n)$.

We will make use of the simple, but crucial, observation from \cite{NPSW20} that $A$ and $k\cdot B$ have a full sumset for all $k\ge 1$. 
\begin{lemma}[{\cite[Lemma~4.2]{NPSW20}}]\label{lem:injective}
The map $(\vec{a},\vec{c})\mapsto\vec{a}+\vec{c}$ from $A\times(k\cdot B)$ to $A + k\cdot B$ is injective.
\end{lemma}
\begin{proof}
Indeed, if $\vec{a}_{1} + (\vec{b}^{(1)}_1 + \dots + \vec{b}^{(1)}_k) = \vec{a}_{2} + (\vec{b}^{(2)}_1 + \dots + \vec{b}^{(2)}_k$), where $\vec{a}_i \in A$ and $\vec{b}^{(i)}_j \in B$, then taking the inner product of both sides with $\vec{w}$ and using $\sang{\vec{w}, \vec{b}} = \tau$ for all $b \in B$, we see that $\sang{\vec{w}, \vec{a}_1} = \sang{\vec{w}, \vec{a}_2}$, which implies that $\vec{a}_1 = \vec{a}_2$ by the definition of $A$. 
\end{proof}

We are now ready to prove \cref{thm:main}.

\begin{proof}[Proof of \cref{thm:main}]
Let $k\ge 2$ be a parameter which will be chosen later depending on $\epsilon$. We may assume $\epsilon\in(0,(2C_{\ref{lem:sup-ratio}})^{-2})$ by adjusting $C_{\ref{thm:main}}$ appropriately at the end to make larger values trivial. By \cref{lem:injective}, for each  $\vec{x}\in\{0,\ldots,k+1\}^n$ for which there exist $\vec{a} \in A$ and $\vec{c} \in k\cdot B$ with $\vec{a}+\vec{c} = \vec{x}$, there exists a unique such choice $\vec{a} = \vec{a}(\vec{x}) \in A$. (For $\vec{x}\notin A+k\cdot B$, we let $\vec{a}(\vec{x})$ be an arbitrary element of $A$.) 

Now, let $\vec{a}$ be uniform on $A$, let $\vec{b}_1,\ldots,\vec{b}_k$ be uniform on $B$, and let $\vec{v}_1,\ldots,\vec{v}_k$ be uniform on $\{0,1\}^n$. Let $C_i\subseteq\{0,\ldots,k+1\}^n$ be the set of vectors with $i$ coordinates equal to $k+1$. For $\vec{x}\in\{0,\ldots,k+1\}^n$, we let $\vec{x}^\ast \in \{0,\dots,k\}^{n}$ denote the vector obtained by setting every occurrence of $k+1$ in $\vec{x}$ to $k$. We have
\begin{align*}
1 &= \mb{P}[\vec{a}+\vec{b}_1+\cdots+\vec{b}_k\in\{0,\ldots,k+1\}^n]\\
&= \sum_{i=0}^n\sum_{\vec{x}\in C_i}\mb{P}[\vec{a}+\vec{b}_1+\cdots+\vec{b}_k = \vec{x}]\\
&\le \sum_{i=0}^n\sum_{\vec{x}\in C_i}\mb{P}[\vec{a} = \vec{a}(\vec{x})]\mb{P}[\vec{b}_1+\cdots+\vec{b}_k = \vec{x}-\vec{a}(\vec{x})]\\
&\le\frac{1}{|A|}\sum_{i=0}^n\sum_{\vec{x}\in C_i}\bigg(\frac{2^n}{|B|}\bigg)^k\mb{P}[\vec{v}_1+\cdots+\vec{v}_k = \vec{x}-\vec{a}(\vec{x})]\\
&\le\frac{e^{k\epsilon n}}{|A|}\sum_{i=0}^n\sum_{\vec{x}\in C_i}\mb{P}[\vec{v}_1+\cdots+\vec{v}_k = \vec{x}^\ast]\sup_{\vec{a}\in A}\frac{\mb{P}[\vec{v}_1+\cdots+\vec{v}_k = \vec{x}-\vec{a}]}{\mb{P}[\vec{v}_1+\cdots+\vec{v}_k = \vec{x}^\ast]}\\
&= \frac{e^{k\epsilon n}}{|A|}\sum_{i=0}^n(1/2^k)^i\sum_{S\in\binom{[n]}{i}}\mb{E}_{\vec{x}\sim\on{Bin}(k)^{\otimes ([n]\setminus S)}\times\{k+1\}^{S}}\bigg[\sup_{\vec{a}\in A}\frac{\mb{P}[\vec{v}_1+\cdots+\vec{v}_k = \vec{x}-\vec{a}]}{\mb{P}[\vec{v}_1+\cdots+\vec{v}_k = \vec{x}^\ast]}\bigg].
\end{align*}
Let $A_S$ be the set of elements in $A \subseteq \{0,1\}^{n}$ whose support contains $S$. Let
\[A'_S = \{\vec{a}' \in \{0,1\}^{[n]\setminus S}: \exists \vec{a} \in A_S \text{ with }\vec{a}|_{[n]\setminus S} = \vec{a}'\}.\]
Recall that $|A| = \exp(\delta n)$. Abusing notation so that the supremum of an empty set is $0$, we can continue the above chain of inequalities to get that
\begin{align*}
1&\le\frac{e^{k\epsilon n}}{|A|}\sum_{i=0}^n(1/2^k)^i\sum_{S\in\binom{[n]}{i}}\mb{E}_{\vec{x}\sim\on{Bin}(k)^{\otimes([n]\setminus S)}\times\{k+1\}^{S}}\bigg[\sup_{\vec{a}\in A}\frac{\mb{P}[\vec{v}_1+\cdots+\vec{v}_k = \vec{x}-\vec{a}]}{\mb{P}[\vec{v}_1+\cdots+\vec{v}_k = \vec{x}^\ast]}\bigg]\\
&= \frac{e^{k\epsilon n}}{|A|}\sum_{i=0}^n(1/2^k)^i\sum_{S\in\binom{[n]}{i}}\mb{E}_{\vec{x}\sim\on{Bin}(k)^{\otimes([n]\setminus S)}}\bigg[\sup_{\vec{a}\in A_S'}\frac{\mb{P}[\vec{v}_1+\cdots+\vec{v}_k = (\vec{x}-\vec{a})\times \{k\}^{S}]}{\mb{P}[\vec{v}_1+\cdots+\vec{v}_k = \vec{x} \times \{k\}^{S}]}\bigg]\\
&= \frac{e^{k\epsilon n}}{|A|}\sum_{i=0}^n(1/2^k)^i\sum_{S\in\binom{[n]}{i}}\mb{E}_{\vec{x}\sim\on{Bin}(k)^{\otimes([n]\setminus S)}}\bigg[\sup_{\vec{a}\in A_S'}\frac{\mb{P}[(\vec{v}_1+\cdots+\vec{v}_k)|_{[n]\setminus S} = \vec{x}-\vec{a}]}{\mb{P}[(\vec{v}_1+\cdots+\vec{v}_k)|_{[n]\setminus S} = \vec{x}]}\bigg]\\
&\le \frac{e^{k\epsilon n}}{|A|}\left(\sum_{i=0}^{n/2}\cdot + \sum_{i=n/2}^{n}2^{-ki}\cdot 2^{n}\cdot \left(\max_{\ell}\frac{\max\Big\{\binom{k}{\ell- 1},\binom{k}{\ell}\Big\}}{\binom{k}{\ell}}\right)^{n-i}\right)\\
&\le \frac{e^{k\epsilon n}}{|A|}\left(\sum_{i=0}^{n/2}\cdot + n\cdot 2^{-kn/2}\cdot 2^{n}\cdot k^{n}\right)\\
&\le\frac{e^{k\epsilon n}}{|A|}\bigg(\sum_{i=0}^{n/2}\binom{n}{i}2^{-ki}\exp(C_{\ref{lem:sup-ratio}}(k^{-1}+(2\delta)^{1/2}k^{-1/2})(n/2)) + 2^{-kn/4}\bigg)\\
&\le\exp(-\delta n)\exp\left(O(k\epsilon + k^{-1} + \delta^{1/2}k^{-1/2})n\right)
\end{align*}
by \cref{lem:sup-ratio} applied to $A_S$, as long as $n/2\ge k\ge C_{\ref{lem:sup-ratio}} \ge 20$. To deduce the last line, note that $\binom{n}{i}2^{-ki}\le (2^{-k}en/i)^i$, so for $i\ge\lceil en/2^{k-1}\rceil$ the sum of weighted binomials is bounded by a geometric series. Additionally, for $1\le i\le en/2^k$, if this interval is nonempty, the sum of binomials is certainly bounded by $\exp(O(k^{-1}n))$.

Hence, the above inequality yields
\[\delta\le C(k\epsilon + k^{-1}+ \delta^{1/2}k^{-1/2})\]
for some absolute constant $C > 0$. Now letting $k = \epsilon^{-1/2}/2$ (note that this satisfies $2C_{\ref{lem:sup-ratio}}\le 2k = \epsilon^{-1/2} \le n$), we find that
\[\delta = O(\epsilon^{1/2}),\]
 as desired.
\end{proof}

The proof of \cref{lem:sup-ratio} relies on the following preliminary estimate. 
\begin{lemma}\label{lem:initial-bound}
If $1\le s\le k/(16\pi)$, then \[\mb{E}_{x\sim\on{Bin}(k)}\bigg(\frac{x}{k+1-x}\bigg)^s\le\exp(10\pi s^2/k)+2k^s(4/5)^k.\]
\end{lemma}
\begin{proof}
We let $x\sim\on{Bin}(k)$ and $y = x-k/2\sim\on{Bin}(k)-k/2$ throughout. We let $z\sim\mc{N}(0,k\pi/8)$. 
We have
\begin{align*}
\mb{E}_{x\sim\on{Bin}(k)}\bigg[\bigg(\frac{x}{k+1-x}\bigg)^s\bigg]&=\mb{E}_y\bigg[\bigg(1 + \frac{2y-1}{k/2+1-y}\bigg)^s\bigg]\\
&\le \mb{E}_y\bigg[\bigg(1+\frac{2y}{k/2+1-y}\bigg)^s\mbm{1}_{|y|\le k/3}\bigg] + k^s\mb{P}[|y|\ge k/3]\\
&\le\mb{E}_y\bigg[\bigg(1+\frac{2y}{k/2+1-y}\bigg)^s\mbm{1}_{|y|\le k/3}\bigg] + 2k^s(4/5)^k.
\end{align*}
Note that the probability estimate for $\mb{P}[\mbm{1}_{|y|\ge k/3}]$ follows from the sharp (entropy) version of the Chernoff-Hoeffding theorem. Since for $|y| \le k/3$,
\[\frac{2y}{(k/2 + 1-y)} \le \frac{2y}{k/2+1} + \frac{8y^{2}}{(k/2+1)^{2}},\]
and using $(1+x) \le \exp(x)$, we can continue the previous inequality as
\begin{align*}
\mb{E}_{x\sim\on{Bin}(k)}\bigg[\bigg(\frac{x}{k+1-x}\bigg)^s\bigg]&\le\mb{E}_y\bigg[\bigg(1+\frac{2y}{k/2+1}+\frac{8y^2}{(k/2+1)^2}\bigg)^s\mbm{1}_{|y|\le k/3}\bigg] + 2k^s(4/5)^k\\
&\le\mb{E}_y\bigg[\exp\bigg(\frac{4sy}{k+2}+\frac{32sy^2}{k^2}\bigg)\bigg] + 2k^s(4/5)^k.
\end{align*}
Now, let $z_1,\dots,z_k$ be i.i.d.\ $\mc{N}(0,1)$ random variables. Then,
\[y \sim \frac{1}{2}\left(\on{sgn}{z_1} + \dots + \on{sgn}{z_k}\right).\]
Moreover, for any $-k \le \ell \le k$,
\[\mb{E}[z_1 + \dots + z_k \mid \on{sgn}(z_1) + \dots + \on{sgn}(z_k) = \ell] = \sqrt{\frac{2}{\pi}}\ell.\]
In particular, under this coupling of $y,z_1,\dots,z_k$, we have
\[\mb{E}[z_1+\dots + z_k \mid y] = \sqrt{\frac{8}{\pi}}y.\]
Let $z = z_1 + \dots +z_k$, so that $z \sim \mc{N}(0,k)$.
Then, by the convexity of
\[f(y)= \exp\bigg(\frac{4sy}{k+2}+\frac{32sy}{k^2}\bigg)\]
and using Jensen's inequality, we have
\begin{align*}
\mb{E}_{y}f(y)
&= \mb{E}_{y, z_1, \dots, z_k}f(y)\\
&= \mb{E}_{y,z_1,\dots,z_k} f\left(\sqrt{\frac{\pi}{8}}\mb{E}[z \mid y]\right)\\
&\le \mb{E}_{z}f(\sqrt{\pi}z/\sqrt{8})\\
&= \mb{E}_{w\sim\mc{N}(0,1)}\exp\bigg(\frac{s\sqrt{2k\pi}}{k+2}w+\frac{4s\pi}{k}w^2\bigg)\\
&= \bigg(1-\frac{8\pi s}{k}\bigg)^{-1/2}\exp\bigg(\frac{\pi s^2k^2}{(k+2)^2(k-8\pi s)}\bigg)\\
&\le\exp\bigg(\frac{8\pi s}{k} + \frac{2\pi s^2}{k}\bigg) \\
&\le\exp(10\pi s^2/k).\qedhere
\end{align*}
\end{proof}

Finally, we can prove \cref{lem:sup-ratio}
\begin{proof}[Proof of \cref{lem:sup-ratio}]
We may assume that $\delta \ge 2000/k$ since the statement for $\delta < 2000/k$ follows from the statement for $\delta = 2000/k$. Also, note that we may assume that $\delta \le \log2$.
For any $t \in \mb{R}$, we have
\begin{align*}
\mb{P}_{\vec{x}}\bigg[\sup_{\vec{a}\in A}\frac{\mb{P}_{\vec{b}}[\vec{b} = \vec{x}-\vec{a}]}{\mb{P}_{\vec{b}}[\vec{b} = \vec{x}]}\ge e^{tn}\bigg]&\le |A|\sup_{\vec{a}\in A}\mb{P}_{\vec{x}}\bigg[\frac{\mb{P}_{\vec{b}}[\vec{b} = \vec{x}-\vec{a}]}{\mb{P}_{\vec{b}}[\vec{b} = \vec{x}]}\ge e^{tn}\bigg]\\
&\le |A|\sup_{\vec{a}\in A}\inf_{s\ge 2}\exp(-stn)\mb{E}_{\vec{x}}\bigg[\bigg(\frac{\mb{P}_{\vec{b}}[\vec{b} = \vec{x}-\vec{a}]}{\mb{P}_{\vec{b}}[\vec{b} = \vec{x}]}\bigg)^s\bigg]\\
&= |A|\sup_{\vec{a}\in A}\inf_{s\ge 2}\exp(-stn)\prod_{i=1}^n\mb{E}_{x\sim\on{Bin}(k)}\bigg[\bigg(\frac{\mb{P}[\on{Bin}(k) = x-a_i]}{\mb{P}[\on{Bin}(k) = x]}\bigg)^s\bigg]\\
&\le |A|\inf_{s\ge 2}\exp(-stn)\bigg(\mb{E}_{x\sim\on{Bin}(k)}\bigg(\frac{x}{k+1-x}\bigg)^s\bigg)^n.
\end{align*}
In the last line, we have used that
\begin{align*}
    \mb{E}_{x\sim\on{Bin}(k)}\bigg[\bigg(\frac{x}{k+1-x}\bigg)^s\bigg]
    &\ge \bigg(\mb{E}_{x\sim\on{Bin}(k)}\bigg[\frac{x^2}{(k+1-x)^2}\bigg]\bigg)^{s/2}\\
    &= \bigg(\sum_{\ell = 0}^{k-1}\frac{\ell+1}{k-\ell}\binom{k}{\ell}2^{-k}\bigg)^{s/2}\\
    &= \bigg(\sum_{\ell = 0}^{k-1}\bigg(\frac{k+2}{k} + \frac{4(k+1)(\ell-k/2)}{k^2} + \frac{(k+1)(k-2\ell)^2}{k^2(k-\ell)}\bigg)\binom{k}{\ell}2^{-k}\bigg)^{s/2}\\
    &\ge \bigg(\sum_{\ell = 0}^{k-1}\bigg(\frac{k+2}{k} + \frac{4(k+1)(\ell-k/2)}{k^2}\bigg)\binom{k}{\ell}2^{-k}\bigg)^{s/2}\\
    &=\bigg(\frac{k+2}{k}-\frac{3k+4}{k}2^{-k}\bigg)^{s/2}\\
    &\ge 1\\
\end{align*}
if $k\ge 3$.
Therefore, by \cref{lem:initial-bound}, we have
\begin{align*}
\mb{P}_{\vec{x}}\bigg[\sup_{\vec{a}\in A}\frac{\mb{P}_{\vec{b}}[\vec{b} = \vec{x}-\vec{a}]}{\mb{P}_{\vec{b}}[\vec{b} = \vec{x}]}\ge e^{tn}\bigg]&\le |A|\inf_{s\ge 2}\exp(-stn)\bigg(\mb{E}_{x\sim\on{Bin}(k)}\bigg(\frac{x}{k+1-x}\bigg)^s\bigg)^n\\
&\le |A|\inf_{2\le s\le k/(16\pi)}\exp(-stn)\bigg(\exp(10\pi s^2/k)+2k^s(4/5)^k\bigg)^n\\
&\le |A|\inf_{2\le s\le k/(10\log k)}\exp(-stn)\bigg(\exp(12\pi s^2/k)\bigg)^n\\
& \le 
\begin{cases}
|A|\exp\left(-\frac{kt^{2}n}{48\pi}\right) &\quad \text{ if } \sqrt{\frac{96\pi\delta}{k}}\le t\le (\log{k})^{-1}\\
|A|\exp\left(-\frac{kn}{48\pi (\log{k})^{2}}\right) &\quad \text{ if } (\log{k})^{-1}\le t \le \log{k}.
\end{cases}
\end{align*}
Here, the second case follows by plugging in $s = k/(24\pi\log{k})$ and simplifying (assuming $C_{\ref{lem:sup-ratio}}$ is large enough so $s\ge 2$), and the first case follows from plugging in $s = kt/(24\pi)$ which satisfies $2 \le s \le k/(10\log{k})$ by the restriction on $t$ and $\delta$.
Finally, since
\[0\le \sup_{\vec{a}\in A}\frac{\mb{P}_{\vec{b}}[\vec{b} = \vec{x}-\vec{a}]}{\mb{P}_{\vec{b}}[\vec{b} = \vec{x}]} \le \left(\max_{\ell}\frac{\max\Big\{\binom{k}{\ell- 1},\binom{k}{\ell}\Big\}}{\binom{k}{\ell}}\right)^{n} \le k^{n},\]
we have
\begin{align*}
\mb{E}_{\vec{x}}\bigg[\sup_{\vec{a}\in A}\frac{\mb{P}_{\vec{b}}[\vec{b} = \vec{x}-\vec{a}]}{\mb{P}_{\vec{b}}[\vec{b} = \vec{x}]}\bigg] 
&= \int_{-\infty}^{\log k}\mb{P}\bigg[\sup_{\vec{a}\in A}\frac{\mb{P}_{\vec{b}}[\vec{b} = \vec{x}-\vec{a}]}{\mb{P}_{\vec{b}}[\vec{b} = \vec{x}]}\ge e^{tn}\bigg]ne^{tn}dt\\
&\le \int_{1/\log{k}}^{\log{k}}\cdot + \int_{\sqrt{96\pi \delta/k}}^{1/\log{k}}\cdot + \int_{-\infty}^{\sqrt{96\pi\delta/k}}ne^{tn}dt\\
&\le e^{\sqrt{96\pi\delta/k}n} + \int_{\sqrt{96\pi\delta/k}}^{1/\log k}|A|\exp\bigg(-\frac{kt^2n}{48\pi}\bigg)ne^{tn}dt\\
&\quad+ \int_{1/\log k}^{\log k}|A|\exp\bigg(-\frac{kn}{48\pi(\log k)^2}\bigg)ne^{tn}dt\\
&\le\exp\left(O(\sqrt{\delta/k})n\right) + \int_{\sqrt{96\pi\delta/k}}^{1/\log{k}} ne^{-tn}dt + 1\\
&\le\exp\left(O(\sqrt{\delta/k})n\right).\qedhere
\end{align*}
\end{proof}

\section*{Acknowledgments} 
The authors are grateful to the anonymous reviewers for several suggestions which helped improve the presentation of the paper.

\bibliographystyle{amsplain}
\bibliography{main.bib}


\begin{aicauthors}
\begin{authorinfo}[vj]
  Vishesh Jain\\
  Department of Statistics, Stanford University\\
  Stanford, California\\
  visheshj\imageat{}stanford\imagedot{}edu \\
  \url{https://jainvishesh.github.io}
\end{authorinfo}
\begin{authorinfo}[as]
  Ashwin Sah\\
  Department of Mathematics, Massachusetts Institute of Technology\\
  Cambridge, Massachusetts\\
  asah\imageat{}mit\imagedot{}edu \\
  \url{http://www.mit.edu/~asah/}
\end{authorinfo}
\begin{authorinfo}[ms]
  Mehtaab Sawhney\\
  Department of Mathematics, Massachusetts Institute of Technology\\
  Cambridge, Massachusetts\\
  msawhney\imageat{}mit\imagedot{}edu \\
  \url{http://www.mit.edu/~msawhney/}
\end{authorinfo}
\end{aicauthors}

\end{document}